\documentclass[article]{amsart}
\usepackage[utf8]{inputenc}

\usepackage{fullpage,hyperref,enumitem}
\usepackage{amsmath,amssymb,mathtools}
\usepackage{float}
\usepackage{pgf}
\usepackage{amsrefs}
\usepackage{amsthm}
\usepackage{mathrsfs}
\usepackage{tikz}
\usetikzlibrary{calc}

\theoremstyle{plain}
\newtheorem{theorem}{Theorem}

\newtheorem{proposition}{Proposition}

\newtheorem{lemma}[proposition]{Lemma}

\theoremstyle{definition}
\newtheorem{definition}{Definition}
\newtheorem{example}{Example}
\newtheorem{remark}[proposition]{Remark}

\newcommand{\seqnum}[1]{\href{https://oeis.org/#1}{\rm \underline{A#1}}}

\usepackage{thmtools} 
\usepackage{thm-restate}

\title{The Combinatorics of Multi-Lane Merging}

\author[Hiveley]{Aurora Hiveley}
\address[A.~Hiveley]{Department of Mathematics, Rutgers University, Piscataway, NJ 08854}
\email{\textcolor{blue}{\href{mailto:aurora.hiveley@rutgers.edu}{aurora.hiveley@rutgers.edu}}}

\author[Zeilberger]{Doron Zeilberger}
\address[D.~Zeilberger]{Department of Mathematics, Rutgers University, Piscataway, NJ 08854}
\email{\textcolor{blue}{\href{mailto:doronzeil@gmail.com}{doronzeil@gmail.com}}}

\begin{document}
\begin{abstract}
We extend the nice treatment of V. Bardenova, E. Insko, K. Johnson, and S. Sullivan of the combinatorics of two-lane mergings to the multi-lane case.

\end{abstract}

\maketitle

\section{Introduction} \label{sec:intro}
Consider a road with two lanes of traffic. The road approaches a stoplight, and shortly after the stoplight the left lane will merge into the right lane. An individual driving a car on this road must make a decision as they approach the stoplight: which lane will they occupy? The driver may prefer to stay in the right lane so that they need not worry about merging, or they may prefer to occupy the lane with fewer cars so that they can make it through the stoplight as quickly as possible. This problem is referred to as a combinatorial lane merging problem, and it has been studied previously by \cite{bard} and \cite{insko}.

In this problem, we will represent each car's preference with a single digit: a 1 if the car prefers the right lane, or a 2 if the car prefers the shorter of the two lanes. In the case of a tie, a car denoted by a 2 will opt for the right lane, so the right lane is always at least as long as the left lane as a result. Observe, then, that at any point during the resolution of the arrival sequence, the lane lengths are weakly decreasing from right to left. A car can only join the left lane if the right lane is strictly longer, so the left lane can never be longer than the right lane. 

\begin{definition}
A 2-lane \textit{arrival sequence} is a binary sequence of 1's and 2's representing the preferences of all cars in a lane merging problem.
\end{definition}

\begin{example} \label{ex:arrseq}
Let $A$ be the arrival sequence $A = 11222121$. Car 1's preference is the right lane, so Car 1 moves to the right lane. Car 2's preference is also the right lane, so it also joins the right lane. Car 3 prefers the shorter of the two lanes, and since no cars are in the left lane, Car 3 joins the left lane. The same occurs for Car 4. Then Car 5 prefers the shortest lane, but both lanes have two cars in them, so Car 5 opts for the right lane to resolve the tie. Next, Car 6 moves to the right lane, and Car 7 will go to the left lane since there are only two cars there compared to four cars in the right lane. Lastly, Car 8 goes to the right lane. Then the final configuration is $[1,2,5,6,8]$ in the right lane and $[3,4,7]$ in the left.
\end{example}

Bardenova et al. \cite{bard} is primarily concerned with computing the expected length of the right lane for an arrival sequence of length $n$ (of course, the length of the left then follows immediately). In their work, they represent a lane merging problem with a 2-dimensional lattice path, where the vertical component is incremented when a car joins the right lane and the horizontal component is incremented when a car joins the left lane. Then an arrival sequence with $y$ cars in the right lane and $x$ cars in the left lane will end at the point $(x,y)$. By counting these lattice paths, the authors were able to determine a recurrence and eventually a closed formula for the number of cars $y$ in the right lane after the sequence has fully resolved. From there, computing the expectation and studying the asymptotics is straightforward. Additional analysis of the lane merging problem with limits on the capacity of each lane was also discussed in \cite{insko}. 

The authors of \cite{bard} also propose a generalization of their work which considers three lanes of traffic. They suggest extending the definition of an arrival sequence as follows:

\begin{itemize}
    \item A 1 denotes a car whose preference is the right lane. 
    \item A 2 denotes a car whose preference is the shorter of the two rightmost lanes.
    \item A 3 denotes a car whose preference is the shortest of all three lanes.
\end{itemize}

And, in fact, we can generalize this extension even further to the case of $k$ lanes of traffic. We will define an arrival sequence like so:

\begin{definition} \label{def:gen_arr_seq}
    An \textit{arrival sequence} is a sequence in $\{1,\dots, k\}^n$ where each digit represents a car's preference. A 1 represents a car which will only enter the right lane, a 2 will enter the shortest of the two rightmost lanes (opting for the right in case of a tie), and generally an $i$ will enter the shortest of the $i$ rightmost lanes, opting for the rightmost tied lane in the case of a tie.
\end{definition}

This generalization will be the focus of our work, and henceforth we will always assume that an arrival sequence has $k$ lanes and length $n$. Now that our sequences have more than two lanes, rather than focusing on the expected length of a specific lane, we will instead calculate what would be the endpoint of a sequence's $k$-dimensional lattice path. We will define the final configuration of an arrival sequence as follows.

\begin{definition}
    An arrival sequence \textit{resolves} to the \textit{end state} $(x_1,x_2, \dots, x_k)$ if there are $x_1$ cars in the right lane, $x_2$ cars in the second lane, etc. after all cars have merged according to their preferences.
\end{definition}

Then our primary goal in this paper is to count the number of arrival sequences which resolve to a given endstate, $(x_1, x_2, \dots, x_k)$. Certainly if we can perform these calculations, we may also compute the statistic such as the expected length of an individual lane. Our main result is given in Theorem \ref{thm:arr_cf}.

\newtheorem*{thm:main}{Theorem \ref{thm:arr_cf}}
\begin{thm:main}
The number of $k$-lane arrival sequences with end state $(x_1, x_2, \dots, x_k)$ (where $x_i$ is the number of cars in the $i$-th lane from the right) are counted by the closed form
\[ F(x_1, \dots, x_k) = \binom{n}{x_1, \dots, x_k} \cdot \binom{k}{p_1, \dots, p_r} \]
where $p_1, p_2, \dots, p_r$ is the partition of $k$ defined by $(x_1, \dots, x_k)$ in frequency notation.
\end{thm:main}

In this paper, we will begin by considering an alternative visualization of arrival sequences using standard Young tableaux (henceforth called SYT) rather than lattice paths. Then, we will use the properties of SYT to count the number of arrival sequences which map to a given tableau, and then discuss an equinumerous
enumeration of majority voting sequences. This equinumeracy
allows us to derive the above closed formula which counts the number of arrival sequences with a given end state. Finally, we will use this formula to compute some lane length statistics for sequences with ranging numbers of lanes and lengths.

\section{Visualizations with Standard Young Tableaux} \label{sec:syt}

Previously, Bardenova et al. \cite{bard} mapped two-lane arrival sequences to lattice paths in two dimensions. In these paths, the vertical component was incremented whenever a car joined the right lane, and the horizontal component was incremented whenever a car joined the left lane. In this sense, the lattice path never crosses below the diagonal since the left lane can never be longer than the right lane. As we generalize this work to three or more lanes, we note the difficulties visualizing and drawing lattice paths in three or more dimensions. So, we instead pivot to a different visual representation: a standard Young tableau. 

\subsection{Basics of Standard Young Tableau} \label{sub:SYT_basics}

A Young diagram, sometimes called a Ferrers diagram when represented with dots instead of boxes, is a left-justified arrangement of boxes in rows such that no row is longer than any row above it. A standard Young tableau (SYT) is a Young diagram whose boxes have been populated with the integers $\{1, 2, \dots, n\}$ where $n$ is the number of cells in the diagram, each row is increasing from left to right, and each column is increasing from top to bottom \cite{zeil}. Figure \ref{fig:ferrers&syt} depicts a Ferrers diagram and one possible SYT that can be obtained from its population with numbers. The shape, or underlying Ferrers diagram, of a SYT can be written $L = (l_1, l_2, \dots, l_k)$ where $l_i$ is the length of the $i$-th row of the SYT and $l_i \geq l_j$ for all $i < j$. The number of SYT of a given shape is well understood and has been extensively studied in \cite{stanley} and \cite{zeil}, to name just two of many existing sources.

\begin{figure}[h!] 
\centering

\begin{tikzpicture}
\pgfmathsetmacro{\R}{0.35} % set box radius

%%% YOUNG DIAGRAM
% reference node
\node (r) at (0,1.5) {};

% set the nodes in each row 
\foreach \k [count=\ki] in {4,2,1}
\foreach \i in {1,...,\k}
{
% locate node at center of box 
\pgfmathtruncatemacro{\ii}{\i - 1}
\node (r\i) at ($(r)+(\ii*2*\R,-\ki*2*\R)$) {};

% draw box around node
\draw ($(r\i)+(-\R,-\R)$)--($(r\i)+(-\R,\R)$);
\draw ($(r\i)+(-\R,\R)$)--($(r\i)+(\R,\R)$);
\draw ($(r\i)+(\R,\R)$)--($(r\i)+(\R,-\R)$);
\draw ($(r\i)+(\R,-\R)$)--($(r\i)+(-\R,-\R)$);
}

\node (t) at (3.5,0) {$\implies$};

%%% STANDARD YOUNG TABLEAU
% cell labels, read across each row 
\def\array{{{1,3,5,6},{2,7},{4}}}

% reference node
\node (s) at (5,1.5) {};

% set the nodes in each row 
\foreach \k [count=\ki] in {4,2,1}
\foreach \i in {1,...,\k}
{
% locate node at center of box 
\pgfmathtruncatemacro{\ii}{\i - 1}
\pgfmathtruncatemacro{\kii}{\ki-1}

\node (s\i) at ($(s)+(\ii*2*\R,-\ki*2*\R)$) {\pgfmathparse{\array[\kii][\ii]}\pgfmathresult};

% draw box around node
\draw ($(s\i)+(-\R,-\R)$)--($(s\i)+(-\R,\R)$);
\draw ($(s\i)+(-\R,\R)$)--($(s\i)+(\R,\R)$);
\draw ($(s\i)+(\R,\R)$)--($(s\i)+(\R,-\R)$);
\draw ($(s\i)+(\R,-\R)$)--($(s\i)+(-\R,-\R)$);
}

\end{tikzpicture}
\caption{An example Young diagram on $n=7$ (left) and the same Young diagram filled in to form a standard Young tableau (right) of shape $L=(4,2,1)$.}
\label{fig:ferrers&syt}
\end{figure}
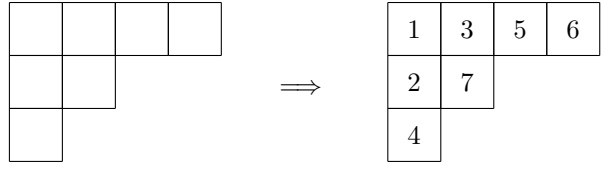

Since the lanes in our arrival sequences must all be weakly decreasing in length from right to left, the rows of a SYT are ideal for representing the final configuration of lanes in an arrival sequence. The cars arrive at the intersection in numerical order, so of course each row will be increasing. Furthermore, a car can only join a lane further to the left if there are at least that many cars in all lanes to the right, so the columns will also be increasing. We also note that using a SYT to represent an arrival sequence offers greater ease of visualization since an additional lane of traffic is modeled with a new row of the tableau rather than a new dimension for the lattice path to traverse.

\subsection{Tableau Enumeration} \label{sub:syt_enumer}
The original map from two-lane arrival sequences to lattice paths was not one-to-one. For example, regardless of what the first car's preference is, it will always join the right lane. Then there are at least two sequences which map to the same lattice path: the sequence $1 + S$ and the sequence $2 + S$, where $S$ is the rest of the sequence. The precise enumeration of how many sequences map to a lattice path is detailed in \cite{bard} in an effort to aid the enumeration of arrival sequences' ending configurations. 

Our new mapping to SYT is no different: there are still several arrival sequences which map to a single SYT. Crucially, the multiplicity of each visualization is quantified by considering each time that a tie break occurs. In Bardenova et al. \cite{bard}, the authors define a ``bounce" as a place where the lattice path makes contact with the diagonal and a tie break forces the lattice path to move upwards, effectively ``bouncing" off of the main diagonal. In more general terms, we say:

\begin{definition}
    A \textit{bounce} is a location in the arrival sequence where a car with preference $i > 1$ opts for a lane further to the right of $i$ because of a tie between lane lengths. 
\end{definition}

In the language of SYT, we see that a ``bounce" occurs when two or more rows have the same length, and the $n$-th car is added to the highest-up row with a tied length. Consider the resolution of the arrival sequence $12131233$ illustrated in Figure \ref{fig:bounceSYT}. While Car 8 is willing to join whichever lane is the shortest, since lanes 2 and 3 are tied in length with two cars each, Car 8 ``bounces" and lands in lane 2.

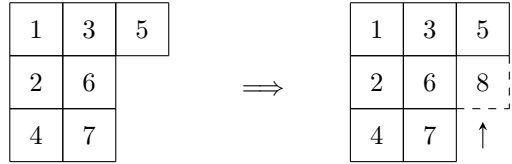
\begin{figure}[h!]
\begin{tikzpicture}
\pgfmathsetmacro{\R}{0.35} % set box radius

%%% STANDARD YOUNG TABLEAU
% cell labels, read across each row 
\def\array{{{1,3,5},{2,6},{4,7}}}

% reference node
\node (r) at (0,1.5) {};

% set the nodes in each row 
\foreach \k [count=\ki] in {3,2,2}
\foreach \i in {1,...,\k}
{
% locate node at center of box 
\pgfmathtruncatemacro{\ii}{\i - 1}
\pgfmathtruncatemacro{\kii}{\ki-1}

\node (r\i) at ($(r)+(\ii*2*\R,-\ki*2*\R)$) {\pgfmathparse{\array[\kii][\ii]}\pgfmathresult};

% draw box around node
\draw ($(r\i)+(-\R,-\R)$)--($(r\i)+(-\R,\R)$);
\draw ($(r\i)+(-\R,\R)$)--($(r\i)+(\R,\R)$);
\draw ($(r\i)+(\R,\R)$)--($(r\i)+(\R,-\R)$);
\draw ($(r\i)+(\R,-\R)$)--($(r\i)+(-\R,-\R)$);
}

\node (t) at (3,0) {$\implies$};

%%% STANDARD YOUNG TABLEAU
% cell labels, read across each row 
\def\array{{{1,3,5},{2,6},{4,7}}}

% reference node
\node (s) at (4.5,1.5) {};

% set the nodes in each row 
\foreach \k [count=\ki] in {3,2,2}
\foreach \i in {1,...,\k}
{
% locate node at center of box 
\pgfmathtruncatemacro{\ii}{\i - 1}
\pgfmathtruncatemacro{\kii}{\ki-1}

\node (s\i) at ($(s)+(\ii*2*\R,-\ki*2*\R)$) {\pgfmathparse{\array[\kii][\ii]}\pgfmathresult};

% draw box around node
\draw ($(s\i)+(-\R,-\R)$)--($(s\i)+(-\R,\R)$);
\draw ($(s\i)+(-\R,\R)$)--($(s\i)+(\R,\R)$);
\draw ($(s\i)+(\R,\R)$)--($(s\i)+(\R,-\R)$);
\draw ($(s\i)+(\R,-\R)$)--($(s\i)+(-\R,-\R)$);
}

\node (ss) at ($(s) + (4*\R,-4*\R)$) {8}; 
\draw [dashed] ($(s) + (3*\R,-5*\R)$)--($(s) + (5*\R,-5*\R)$);
\draw [dashed] ($(s) + (5*\R,-3*\R)$)--($(s) + (5*\R,-5*\R)$);

\draw [stealth-] ($(s) + (4*\R,-5.5*\R)$)--($(s) + (4*\R,-6.5*\R)$);

\end{tikzpicture}
\caption{Resolution of Car 8 in the arrival sequence $12131233$, which ``bounces" into lane 2.}
\label{fig:bounceSYT}
\end{figure}

Then to count the number of arrival sequences $\#A(T)$ which resolve to a single SYT $T$, we must keep track of the number of bounces which occur during its resolution, and the number of rows which are tied at each bounce. 

\begin{lemma} \label{lem:bounce}
    If $T$ is a SYT and $T'$ is the SYT obtained by deleting the largest element, $n$, from cell $c$ of $T$, then
\[ \frac{\#A(T)}{\#A(T')} = \beta(c)\]
where
    $$\beta(c) := \begin{cases}
    \# \left\{ c' = [i', j-1] \mid i' > i, T[c'] < T[c]  \right\} & j > 1 \\
    k - i & j = 1
    \end{cases}$$
\end{lemma}

\begin{proof}
    Note that $\#A(T)/\#A(T')$ is equal to the number of possible preferences that could be in position $n$ of the arrival sequence that would produce this SYT. Let $n$ be in cell $[i,j]$ of tableau $T$, meaning that Car $n$ joined lane $i$ in the arrival sequence for $1 \leq i \leq k$. It suffices to count the number of possible preferences could Car $n$ have had which cause it to join lane $i$.
    
    Certainly if Car $n$'s preference is any digit smaller than $i$, then lane $i$ would never be considered by the car, so we can narrow our search space to $\{i, i+1, \dots, k\}$. However, a preference of $i+1$ will only cause Car $n$ to join lane $i$ if the lanes $i$ and $i+1$ have tied lane lengths. In the tableau $T$, this means that there are already digits occupying rows $i$ and $i+1$ in the column to the left of the element $n$. Extrapolating to lower rows, we see that a preference of $i + r$ will only cause Car $n$ to join lane $i$ if \textit{all} lanes $\{i, i+1, \dots, i+r\}$ have tied lengths, so all rows $\{i, i+1, \dots, i+r\}$ of $T$ have entries in the column to the left of the box containing $n$. Then the number of possible preferences that Car $n$ could have which would cause it to join lane $i$ is equal to $\beta(c)$ where $c = [i,j]$ is the location of Car $n$. In this sense, $\beta(c)$ counts the number of dimensions that the merging path would ``bounce" in, or the number of possible digits that Car $n$ could have that would still cause it to enter lane $i$ and thus fill location $[i,j]$ in $T$. Once we have determined the number of possible digits in position $n$ of an arrival sequence which resolves to $T$, we can ``peel off" $n = T[c]$ from $T$, and consider all possible preferences that Car $n-1$ can have in the same fashion. 
    
    Note, however, that any entry in the first column of $T$ has no column to the left to compare to. But the first car to join any lane can prefer any lane to the left, and it will still join lane $i$. For example, as noted before, Car 1 can prefer any of the $k$ lanes, and it will always join the first lane. If Car 2 joins the second lane, then it could have preferred any of the $k-1$ open lanes to the left, and it would still join lane 2. In general, a car in the first column and the $i$-th row could have $k-i$ preferences corresponding to the empty lanes to the left. Hence for $c = [i,j]$ when $j=1$, $\beta(c)$ is defined differently, as if the zero-th column is full.
\end{proof}

\begin{proposition} \label{prop:syt_enumer}
    The number of $k$-lane arrival sequences which resolve to a SYT $T$ with $k$ rows is given by:
    $$\#A(T) = \prod_{c \in T} \beta(c)$$
    where $c = [i,j]$ is a cell in the tableau $T$ and $\beta(c)$ is as defined in Lemma \ref{lem:bounce}.
\end{proposition}

\begin{proof}
    We will count the number of arrival sequences which can be reverse engineered from a tableau $T$. Start first with the final cell added to the tableau, or the final car to choose a lane: $n$ in cell $c$. By Lemma \ref{lem:bounce}, there are $\beta(c)$ possible $n$-th digits of the arrival sequence which would produce the tableau $T$. 

    We then repeat this process for the largest cell of $T'$, which is $n-1$ in cell $c'$, and so on. At each iteration, we peel off a cell $c^\ast$ and use $\beta(c^\ast)$ to count the number of preferences that that car could have. After repeating this for all cells $c$ (and thus all cars), we have that the number of possible arrival sequences can be obtained from the product of the number of possible preferences for each car, which is $\prod_{c \in T} \beta (c)$. 
\end{proof}

\begin{example}
Consider the SYT below. We calculate $\beta(c)$ for each of the five cells as follows:

\begin{center}
\begin{tikzpicture}
\pgfmathsetmacro{\R}{0.35} % set box radius
%%% STANDARD YOUNG TABLEAU
% cell labels, read across each row 
\def\array{{{1,3,5},{2},{4}}}

% reference node
\node (r) at (0,0) {};

% set the nodes in each row 
\foreach \k [count=\ki] in {3,1,1}
\foreach \i in {1,...,\k}
{
% locate node at center of box 
\pgfmathtruncatemacro{\ii}{\i - 1}
\pgfmathtruncatemacro{\kii}{\ki-1}

\node (r\i) at ($(r)+(\ii*2*\R,-\ki*2*\R)$) {\pgfmathparse{\array[\kii][\ii]}\pgfmathresult};

% draw box around node
\draw ($(r\i)+(-\R,-\R)$)--($(r\i)+(-\R,\R)$);
\draw ($(r\i)+(-\R,\R)$)--($(r\i)+(\R,\R)$);
\draw ($(r\i)+(\R,\R)$)--($(r\i)+(\R,-\R)$);
\draw ($(r\i)+(\R,-\R)$)--($(r\i)+(-\R,-\R)$);
}

%%%%%%%%%%%%%%%%%%%%%%%%%%%%%
\node at ($(r) + (-4,-4*\R)$) {\begin{tabular}{c}
$\beta([1,1]) = 3$ \\ $\beta([1,2]) = 2$ \\ $\beta([1,3]) = 1$ \\ 
$\beta([2,1]) = 2$ \\ $\beta([3,1]) = 1$ 
\end{tabular} };
\end{tikzpicture} 
\end{center}

Then $\#A(T) = 3 \cdot 2 \cdot 1 \cdot 2 \cdot 1 = 12$. The twelve 3-lane arrival sequences which resolve to this SYT are:
\[\begin{array}{cccc}
    12131  & 13131 & 12231 & 13231 \\
    22131  & 23131 & 22231 & 23231 \\
    32131  & 33131 & 32231 & 33231 \\
\end{array}\]
\end{example}

\begin{remark}
    The product $\prod_{a \in T} \beta(a)$ where $a = [i,1]$ for some $1 \leq i \leq k$ is:
    $$\prod_{a \in T} \beta(a) = k \cdot (k-1) \cdot \dots \cdot (k-m+1) = k^{\underline{m}}$$
    where $m$ is the number of rows filled in $T$. This means that $\#A(T)$ will always be divisible by $k$. 
    
    Note that there may be empty lanes in an arrival sequence, which translate to empty rows in a SYT. For example, if $A$ is a 3-lane arrival sequence which contains only 1's and 2's, then cars will only merge into one of the two rightmost lanes, and thus the resulting SYT will only have numbers filled in to the top two rows, so $m=2$. However, when determining the number of possible preference digits from a SYT with two occupied rows of three eligible, we must account for the fact that a preference of 3 is allowed, and may even produce the same tableau if a car in the first column has a preference of 3. Hence, we use a falling factorial rather than a standard factorial to account for the possibility of $k$ total rows. This nuance is also reflected in the output of procedures such as \texttt{SYT(k,n)} and \texttt{laneLens(k,A)} in the Maple package linked in the Appendix \ref{sec:appen}.
\end{remark}

\section{Majority Voting Sequences} \label{sec:voting}

Before presenting the statement and proof of Theorem \ref{thm:arr_cf}, we observe that there is a tidy relationship between arrival sequences and another type of sequence.

\begin{definition} \label{def:vote_seq}
    A \textit{majority voting sequence} $V \in \{1,\dots,k\}^n$ is a sequence such that each digit represents a vote cast for one of $k$ total candidates. A \textit{state} of the voting sequence $(p_1, p_2, \dots, p_k)$ is defined such that after $\sum_{i=1}^k p_i$ votes have been cast, there are $p_1$ votes for the most popular candidate, $p_2$ votes for the second most popular, and so on.
\end{definition}

\subsection{SYT for Voting Sequences} \label{sub:voting_syt}
To compare the outcomes of voting sequences with the outcomes of arrival sequences, we seek to also represent voting sequences with SYT. However, the difficulty comes from the fact that a majority sequence does not immediately resolve to a SYT. Consider the following example: 

\begin{example} \label{ex:}
    Let $V$ = 111223333. If we construct an SYT by letting all of the votes for Candidate 1 fall into row 1, and the votes for Candidate 2 into row 2, and so on, then we obtain the leftmost tableau in Figure \ref{fig:voting_syt}. Of course, this \textit{not} a SYT as the row lengths are not weakly decreasing. Even if we sort the rows by length so that they are weakly decreasing, as in the middle tableau of Figure \ref{fig:voting_syt}, the columns will fail to be increasing. 
\end{example}
\begin{figure}[h!]   
\begin{tikzpicture}
\pgfmathsetmacro{\R}{0.35} % set box radius

% cell labels, read across each row 
\def\array{{{1,2,3},{4,5},{6,7,8,9}}}

% reference node
\node (r) at (0,0) {};

% set the nodes in each row 
\foreach \k [count=\ki] in {3,2,4}
\foreach \i in {1,...,\k}
{
% locate node at center of box 
\pgfmathtruncatemacro{\ii}{\i - 1}
\pgfmathtruncatemacro{\kii}{\ki-1}

\node (r\i) at ($(r)+(\ii*2*\R,-\ki*2*\R)$) {\pgfmathparse{\array[\kii][\ii]}\pgfmathresult};

% draw box around node
\draw ($(r\i)+(-\R,-\R)$)--($(r\i)+(-\R,\R)$);
\draw ($(r\i)+(-\R,\R)$)--($(r\i)+(\R,\R)$);
\draw ($(r\i)+(\R,\R)$)--($(r\i)+(\R,-\R)$);
\draw ($(r\i)+(\R,-\R)$)--($(r\i)+(-\R,-\R)$);
}

%%%% 
% cell labels, read across each row 
\def\array{{{6,7,8,9},{1,2,3},{4,5}}}

% reference node
\node (s) at (4.5,0) {};

% set the nodes in each row 
\foreach \k [count=\ki] in {4,3,2}
\foreach \i in {1,...,\k}
{
% locate node at center of box 
\pgfmathtruncatemacro{\ii}{\i - 1}
\pgfmathtruncatemacro{\kii}{\ki-1}

\node (s\i) at ($(s)+(\ii*2*\R,-\ki*2*\R)$) {\pgfmathparse{\array[\kii][\ii]}\pgfmathresult};

% draw box around node
\draw ($(s\i)+(-\R,-\R)$)--($(s\i)+(-\R,\R)$);
\draw ($(s\i)+(-\R,\R)$)--($(s\i)+(\R,\R)$);
\draw ($(s\i)+(\R,\R)$)--($(s\i)+(\R,-\R)$);
\draw ($(s\i)+(\R,-\R)$)--($(s\i)+(-\R,-\R)$);
}

%%%% 
% cell labels, read across each row 
\def\array{{{1,2,3,9},{4,5,8},{6,7}}}

% reference node
\node (t) at (9,0) {};

% set the nodes in each row 
\foreach \k [count=\ki] in {4,3,2}
\foreach \i in {1,...,\k}
{
% locate node at center of box 
\pgfmathtruncatemacro{\ii}{\i - 1}
\pgfmathtruncatemacro{\kii}{\ki-1}

\node (t\i) at ($(t)+(\ii*2*\R,-\ki*2*\R)$) {\pgfmathparse{\array[\kii][\ii]}\pgfmathresult};

% draw box around node
\draw ($(t\i)+(-\R,-\R)$)--($(t\i)+(-\R,\R)$);
\draw ($(t\i)+(-\R,\R)$)--($(t\i)+(\R,\R)$);
\draw ($(t\i)+(\R,\R)$)--($(t\i)+(\R,-\R)$);
\draw ($(t\i)+(\R,-\R)$)--($(t\i)+(-\R,-\R)$);
}

\end{tikzpicture}
\caption{Tableaux obtained by resolving the voting sequence 111223333. From left to right: resolution by candidate voted for, resolution by candidate and sorted by popularity, and resolution by popularity after each vote is cast. }
\label{fig:voting_syt} 
\end{figure}

Then to properly produce a SYT from a voting sequence, we must place voters according to which candidate has the majority number of votes after each vote is cast. In the case of $111223333$, the first three votes are all cast for the majority candidate, vacuously. Then the next two votes are cast for another candidate, who becomes the second most popular candidate. Votes 6 and 7 are then cast for another candidate, who is the least popular, and then tied for the second most popular. Vote 8 is cast for a candidate that now is tied for the most popular, and Vote 9 is cast for the (newly) most popular candidate. This voting sequence produces the rightmost tableau in Figure \ref{fig:voting_syt}.

The tableaux produced by voting sequences have many of the same properties that we observed for tableaux produced by arrival sequences. Let $\#V(T)$ be the number of \textit{voting} sequences which resolve to the tableau, $T$. We present the following two lemmas, which are analogous to Lemmas \ref{lem:bounce} and Proposition \ref{prop:syt_enumer}, respectively.

\begin{lemma} \label{lem:voting_bounce}
If $T$ is a SYT and $T'$ is the SYT obtained by deleting the largest element, $n$, from cell $c$ of $T$,
\[ \frac{\#V(T)}{\#V(T')} = \beta(c)\]
\end{lemma}

\begin{proof}
    Once again, we calculate $\#V(T)/\#V(T')$ by considering cell $c$ containing the number $n$ and determining the number of candidates that Voter $n$ could have voted for. Let $c = [i,j]$, so Voter $n$ voted for the $i$-th most popular candidate, call them $\chi$. This vote could have been cast for any candidate whose popularity was previously tied with $\chi$, i.e., any candidate with the same number of votes in $T'$. Then we must count the number of rows whose length is equal to the length of row $i$ in $T'$, which we know is equal to $\beta(c)$. And, if Voter $n$ was the first to cast a vote for $\chi$, then their vote could have been for any candidate who was yet to receive any votes, which is equal to the number of ``empty" rows of $T'$. Either way, there are $\beta(c)$ total candidates that Voter $n$ could have voted for.
\end{proof}

\begin{lemma} \label{lem:equiv}
    For any SYT $T$ with at most $k$ rows and $n$ boxes, the number of $k$-lane arrival sequences with length $n$ that resolve to $T$ is the same as the number of $k$-candidate voting sequences of length $n$ which resolve to $T$.
\end{lemma}

\begin{proof}
Combining the arguments of Proposition \ref{prop:syt_enumer} Lemma \ref{lem:voting_bounce}, we have that $\#V(T) = \#A(T) = \prod_{c \in T} \beta(c)$. Thus, the number of sequences of both types which resolve to a given tableau $T$ are equinumerous.
\end{proof}

\begin{remark}
    While $\#A(T) = \#V(T)$, meaning that the same \textit{number} of sequences of each type produce $T$, the sets of sequences of each type which resolve to $T$ are usually different. These sets are constructed by the procedures \texttt{LaneSeqsC(T)} and \texttt{VoteSeqsC(T)} from the Maple package in the Appendix \ref{sec:appen}, whose results agree with the (slower) brute-force computations from \texttt{LaneSeqs(T)} and \texttt{VoteSeqs(T)}.
\end{remark}

The specificity of this equinumeracy (not just for tableau shapes but for individual tableaux) points to a bijection between arrival sequences and voting sequences that is implicit in our proof, and is left
to the interested reader to work out \ref{sec:map}.

\section{Closed Form Counting} \label{closed_form}

With Lemma \ref{lem:equiv} in hand, we can conclude that there are equivalent numbers of arrival sequences and of voting sequences which produce SYT of a given shape. Then, a closed form that describes the number of majority voting sequences which produce the end state $(v_1, \dots, v_k)$ will also describe the number of arrival sequences with the same end state. Luckily, we have just that:

\begin{lemma} \label{lem:maj_cf}
    Majority voting sequences are counted by the closed form
    \[ F(v_1, \dots, v_k) = \binom{n}{v_1, \dots, v_k} \cdot \binom{k}{p_1, \dots, p_r} \]
    where $p_1, p_2, \dots, p_r$ is the partition of $k$ defined by $(v_1, \dots, v_k)$ in frequency notation.
\end{lemma}

\begin{example}
  Consider the majority voting sequence ending at $(4,3,3,2)$. Then there are $n = 4 + 3 + 3 + 2 = 12$ votes that are cast for $k = 4$ candidates. Using frequency notation,
  we get $4^1 3^2 2^1$ leading to the partition of $4$, $(2,1,1)$
    Then we can calculate: 
    \[F(4,3,3,2) = \binom{12}{4,3,3,2} \cdot \binom{4}{2,1,1} \]
\end{example}

\begin{proof}
    of \textbf{Lemma \ref{lem:maj_cf}}.

    Say that a total of $n = v_1 + v_2 + \dots + v_k$ votes are cast for $k$ candidates such that the most popular candidate received $v_1$ votes, the second most popular receives $v_2$ votes, and so on. 
    
    There are certainly $\binom{n}{v_1, v_2, \dots, v_k}$ ways that the voters can cast their votes for $k$ candidates so that the ending vote count is precisely $(v_1, v_2, \dots, v_k)$. However, we must also account for the permutations of the candidates by popularity. We convert $(v_1, v_2, \dots, v_k)$ to frequency notation $p_1, p_2, \dots, p_r$. Of these $k$ candidates, we choose $p_1$ of them to have the most votes, or to be tied for the most votes if $p_1 > 1$. Then, of the remaining $k - p_1$, we choose $p_2$ to receive (or be tied for) the second most votes, and so on. This product of binomials is precisely equal to the multinomial $\binom{k}{p_1, p_2, \dots, p_r}$. Then majority sequences with the ending result $(v_1, \dots, v_k)$ can be counted by the desired closed form.
\end{proof}

Observe, of course, that the final vote count $(v_1, v_2, \dots, v_k)$ of a voting sequence can be deduced directly from the frequency of each digit in the voting sequence without actually building the SYT. The same is not true for arrival sequences, where we must go through the process of resolving each car's preference algorithmically to determine the end state. 

Finally, we are equipped to prove Theorem \ref{thm:arr_cf}. First, a reminder of the claim:

\begin{theorem}\label{thm:arr_cf}
The number of $k$-lane arrival sequences with end state $(x_1, x_2, \dots, x_k)$ are counted by the closed form
\[ F(x_1, \dots, x_k) = \binom{n}{x_1, \dots, x_k} \cdot \binom{k}{p_1, \dots, p_r} \]
where $p_1, p_2, \dots, p_r$ is the partition of $k$ defined by $(x_1, \dots, x_k)$ in frequency notation.
\end{theorem}

\begin{proof}
    The number of arrival sequences which produce a given end state $(x_1, x_2, \dots, x_k)$ is equal to the number of SYT with shape $(x_1, x_2, \dots, x_k)$ multiplied by the multiplicty $\#A(T)$ of each tableau with respect to arrival sequences. By Lemma \ref{lem:equiv}, the number of arrival sequences and the number of voting sequences which resolve to a given tableau $T$ are equivalent. Then the number of arrival sequences which resolve to the shape $(x_1, x_2, \dots, x_k)$ is equal to the number of voting sequences which result in the same shape. Then the closed form for majority walks given in Lemma \ref{lem:maj_cf} also functions as a closed form for arrival sequences. 
\end{proof}

Observe that this closed form is equivalent to the equation given in Theorem 6 of \cite{bard}. Their closed form for 2 lanes with lengths $x$ and $y$ where $x > y$ and $x + y = n$ is given as follows:
$$
F'(x,y) = \begin{cases}
    \displaystyle 2 \binom{x+y}{x} & x > y \\ 
    \displaystyle \binom{2x}{x} & x = y \\
\end{cases}
$$

In our case, when $x > y$ we have that $k=2$ and $p_1 = 1$, $p_2 = 1$ since each lane has a unique length. Then $F(x,y) = \binom{x+y}{x,y} \cdot \binom{2}{1,1} = \binom{x+y}{x} \cdot 2$, which matches $F'(x,y)$. Similarly, if $x =y$ in our case, then $p_1 = 2$ when $k = 2$ since each lane has the same length, and $n = 2x$ since $x = y$. So, we have that $F(x,y) = \binom{x+y}{x,y} \cdot \binom{2}{2} = \binom{2x}{x}$. Once again, this matches $F'(x,y)$, so our generalized closed form matches the previous closed form for 2-lane arrival sequences.

\section{Mapping Arrival Sequences and Voting Sequences} \label{sec:map}
In Section 3 of \cite{bard}, the authors provide a bijection between lattice paths and coin flipping sequences. The authors map a lattice path obtained from a two-lane arrival sequence of length $n$ with $x$ cars in the right lane and $y = m-x$ in the left to a sequence of $n$ coin flips where the maximum number of heads or tails is $x$. Their bijection is constructed as follows:

\begin{definition}
    The bijection $\phi: \{1,2\}^n \to \{1,2\}^n$ maps an arrival sequence $A$ to $\phi(A)$ as follows:
    \begin{enumerate}
        \item Define the parity vector $P$ of $A$ as follows:
        \begin{itemize}
            \item [(i)] $P[1] = 1$
    
            \item[(ii)] For $2 \leq i \leq n$, if there is a bounce for Car $i-1$ of $A$, then $A[i] = (A[i-1] \mod 2) + 1$. If there is no bounce for Car $i-1$ of $A$, then $A[i] = A[i-1]$.
        \end{itemize}

        \item Let $C := \phi(A) = A + P \mod 2$.

    \end{enumerate}
\end{definition}

Note that by adding the parity vector $P$ to the existing binary string $A$, whenever a bounce occurs we effectively flip which column of the SYT a car is added to. Consider the following example. 

\begin{figure}[h!]   
\begin{tikzpicture}
\pgfmathsetmacro{\R}{0.35} % set box radius

% cell labels, read across each row 
\def\array{{{1,3,5,7},{2,4,6,8}}}

% reference node
\node (r) at (0,0) {};

% set the nodes in each row 
\foreach \k [count=\ki] in {4,4}
\foreach \i in {1,...,\k}
{
% locate node at center of box 
\pgfmathtruncatemacro{\ii}{\i - 1}
\pgfmathtruncatemacro{\kii}{\ki-1}

\node (r\i) at ($(r)+(\ii*2*\R,-\ki*2*\R)$) {\pgfmathparse{\array[\kii][\ii]}\pgfmathresult};

% draw box around node
\draw ($(r\i)+(-\R,-\R)$)--($(r\i)+(-\R,\R)$);
\draw ($(r\i)+(-\R,\R)$)--($(r\i)+(\R,\R)$);
\draw ($(r\i)+(\R,\R)$)--($(r\i)+(\R,-\R)$);
\draw ($(r\i)+(\R,-\R)$)--($(r\i)+(-\R,-\R)$);
}

\end{tikzpicture}
\caption{SYT obtained by resolving $A$=12221222 or $C$ = 12212112}
\label{fig:bijsyt} 
\end{figure}
\begin{example} \label{ex:bij}
Let $B$ = 12221222. Then the bounces occur for cars numbered 3 and 7, so the parity vector flips at locations 4 and 8, resulting in $P$ = 11122221. Adding $P + A$, we obtain $C$ = 12212112. Figure \ref{fig:bijsyt} depicts the SYT obtained by resolving the arrival sequence $A$. Note that the length of each row is 4, and there are 4 heads and 4 tails in the coin flip sequence $A$. 
\end{example}

Bardenova et al. \cite{bard} suggest that this mapping may be extended to a bijection between lattice paths formed by $k$-lane arrival sequences and sequences of $k$-sided die rolls (or, equivalently, $k$-candidate majority voting sequences).
Such a bijection can be readily constructed (by using recursion) from our proof of the equinumeracy of $A(T)$ and $V(T)$.

\section{Lane Length Statistics} \label{sec:stats}

Previously, Bardenova et al. \cite{bard} computed the expected length of the right lane for a two lane arrival sequence with $n$ cars. More interestingly, the numerator of the expected value produces the following integer sequence, which can be found in the OEIS \cite{oeis} as \seqnum{230137}.  
\[ 0, 2, 6, 18, 44, 110, 252, 588, 1304, 2934, 6380, 14036, 30120, 65260, 138712, \dots \]

Repeating this calculation for three or more lanes, we obtain Table \ref{tab:RKn}. Each of these sequences links back to \seqnum{265080}, which lists the entries of a table which arise from public key cryptography. Henry Bottomley's 2021 comment on this entry also indicates that these table entries count the number of balls in the fullest box if one were to randomly throw $n$ balls into $k$ boxes across all possible outcomes. This hearkens back to our discussion of majority voting, as the balls in boxes problem mimics the same process. 

\begin{table}[h!]
    \centering
    \begin{tabular}{c|ccccccc|c}
     $k$ & $n=1$ & $n=2$ & $n=3$ & $n=4$ & $n=5$ & $n=6$ & $n=7$ & OEIS entry \\
     \hline
     2 & 2 & 6 & 18 & 44 & 110 & 252 & 588 & \seqnum{230137} \\
     3 & 3 & 12 & 51 & 192 & 675 & 2358 & 8043 & \seqnum{265083} \\
     4 & 4 & 20 & 108 & 544 & 2540 & 11544 & 52192 & \seqnum{265084} \\
     5 & 5 & 30 & 195 & 1220 & 7145 & 40230 & 224175 & \seqnum{265085} \\
    \end{tabular}
    \caption{Sum of the length of the right lane in a $k$-lane arrival sequence of length $n$}
    \label{tab:RKn}
\end{table}

The previous authors did not study the expected length of the left lane, likely because it follows immediately from the length of the right lane. However, for three or more lanes, this calculation is less trivial, and can mean two different things: the length of the second lane from the right, or the length of the leftmost lane. In the sense of our majority voting framework, we may be concerned with the number of votes received by the second most popular candidate, or perhaps the number of votes received by the least popular candidate. 

Table \ref{tab:2Kn} tabulates the summed lengths of the second lane from the right across all arrival sequences of length $n$ for $k$ lanes, and Table \ref{tab:LastKn} does the same for the leftmost lane. None of the rows in either Table \ref{tab:2Kn} have entries in the OEIS \cite{oeis}. Of course, the first $k$ numbers in each row of Table \ref{tab:LastKn} match up with the first $k$ entries of the same row in \seqnum{019538}, but the sequences diverge once $n > k$. This makes sense, as \seqnum{019538} counts ordered set partitions of $n$ into $k$ parts, and as long as $n \leq k$, these partitions are equivalent to the end states (SYT) of our arrival sequences. 

\begin{table}[h!]
    \centering
    \begin{tabular}{c|ccccccc}
     $k$ & $n=1$ & $n=2$ & $n=3$ & $n=4$ & $n=5$ & $n=6$ & $n=7$ \\
     \hline
     2 & 0 & 2 & 6 & 20 & 50 & 132 & 308 \\
     3 & 0 & 6 & 24 & 96 & 390 & 1386 & 4830 \\
     4 & 0 & 12 & 60 & 288 & 1500 & 7392 & 33432 \\
     5 & 0 & 20 & 120 & 680 & 4220 &  26220 & 151340 \\
    \end{tabular}
    \caption{Sum of the length of the second lane in a $k$-lane arrival sequence of length $n$ }
    \label{tab:2Kn}
\end{table}
\begin{table}[h!]
    \centering
    \begin{tabular}{c|cccccccc}
     $k$ & $n=1$ & $n=2$ & $n=3$ & $n=4$ & $n=5$ & $n=6$ & $n=7$ & $n=8$ \\
     \hline
     2 & 0 & 2 & 6 & 20 & 50 & 132 & 308 & 744  \\
     3 & 0 & 0 & 6 & 36 & 150 & 630 & 2436 & 8736  \\
     4 & 0 & 0 & 0 & 24 & 240 & 1560 & 8400 & 43344  \\
     5 & 0 & 0 & 0 & 0 & 120 & 1800 & 16800 & 126000 \\
    \end{tabular}
    \caption{Sum of the length of the leftmost lane in a $k$-lane arrival sequence of length $n$}
    \label{tab:LastKn}
\end{table}

\section{Conclusion} \label{sec:conc}

In summary, this paper considered generalized arrival sequences with $k$ lanes. We introduced a mapping from arrival sequences to standard Young tableaux, and we counted the number of sequences which map to each tableau after noting that the map is not one-to-one. Then we produced a closed formula for the number of majority sequences with a given final vote count, and proved that the same closed formula holds for arrival sequences by virtue of their equivalent enumeration of resolution tableaux. Finally, we provided tables counting lane lengths for $k$ lane arrival sequences of length $n$ with references to the OEIS \cite{oeis}, when applicable.

\section*{Appendix} \label{sec:appen}
The findings in this paper are supported by a Maple package linked at each of the following URLs. Any bugs should be reported to the authors at aurora.hiveley@rutgers.edu or doronzeil@gmail.com. 
\begin{itemize}
    \item \url{https://sites.math.rutgers.edu/~zeilberg/tokhniot/MERGING.txt} 
    \item \url{https://aurorahiveley.github.io/merging.txt}
\end{itemize}

An alternative proof of Theorem \ref{thm:arr_cf} using recursion can also be found at each of the following links.
\begin{itemize}
    \item \url{https://sites.math.rutgers.edu/~zeilberg/mamarim/mamarimhtml/merging.html} 
    \item \url{https://aurorahiveley.github.io/merging_appendix.pdf}
\end{itemize}

\end{document}